\documentclass[12,english]{amsart}
\usepackage[T1]{fontenc}
\usepackage{enumerate}
\usepackage{mathrsfs}
\usepackage{hyperref}
\usepackage{color}
\theoremstyle{plain}
\newtheorem{thm}{\protect\theoremname}
\newtheorem{lem}{\protect\lemmaname}
  \theoremstyle{definition}
  \newtheorem{defn}{\protect\definitionname}
  \theoremstyle{plain}
  \newtheorem{assumption}{\protect\assumptionname}
  \theoremstyle{plain}
  \newtheorem{prop}[thm]{\protect\propositionname}
  \theoremstyle{remark}
  \newtheorem{rem}{\protect\remarkname}

\usepackage{thmtools}

\makeatother

\usepackage{babel}
  \providecommand{\assumptionname}{Assumption}
  \providecommand{\definitionname}{Definition}
  \providecommand{\propositionname}{Proposition}
  \providecommand{\remarkname}{Remark}
 \providecommand{\theoremname}{Theorem}
  \providecommand{\lemmaname}{Lemma}
  \newcommand{\A}{\mathbb A}
\newcommand{\E}{\mathbb E}

\newcommand{\D}{\mathbb D}

 \newcommand{\R}{\mathbb R}

 \newcommand{\cE}{\mathcal E}

 \newcommand{\cF}{\mathcal F}

\newcommand{\cH}{\mathcal H}
  \newcommand{\cL}{\mathcal L}

\newcommand{\cP}{\mathcal P}

\newcommand{\cS}{\mathcal S}

 \newcommand{\af}{\alpha}

\newcommand{\ga}{\gamma}

 \newcommand{\de}{\delta}
 
 \newcommand{\om}{\omega}
 \newcommand{\Om}{\Omega}

 \newcommand{\Te}{\Theta}
 \newcommand{\si}{\sigma}

 \newcommand{\mL}{\mathbb L}

\begin{document}

\title[Regularity of weak formulation of MFG]
{On the Regularity of a Weak Formulation of Stochastic Differential Mean-Field Games}

\author[H. S\'anchez Morgado]{H\'ector S\'anchez Morgado}
\email{hector@matem.unam.mx}
\address{Instituto de Matem\'aticas, Universidad Nacional Aut\'onoma de M\'exico,
  Ciudad Universitaria CP 04510, Ciudad de M\'exico, M\'exico.}
\author[J. Sierra]{Jes\'us Sierra}
\email{jesus.sierra@cimat.mx}
\address{Instituto de Matem\'aticas, Universidad Nacional Aut\'onoma de M\'exico,
  Ciudad Universitaria CP 04510, Ciudad de M\'exico, M\'exico.}
\curraddr{CIMAT, De Jalisco s/n, Gto., 36023, Guanajuato, Mexico.}

\begin{abstract}
We study a McKean-Vlasov Forward-Backward Stochastic Differential Equation (FBSDE) in connection with the theory of Stochastic Differential Mean-Field games, particularly the weak (non-fully coupled) formulation described in Section 3.3.1 of \cite{carmona2018probabilistic}. Our main goal is to obtain regularity results for this McKean-Vlasov FBSDE, specifically classical and Malliavin differentiability.
\end{abstract}
\thanks{J. Sierra was supported by CONACYT through its program Estancias Posdoctorales por Mexico}



\maketitle
\section{Introduction}
\label{sec:intro}

\subsection{Mean-field games}
\label{sec:mfg}

A mean-field game (MFG) is a system that models interacting agents
in large populations \cite{cardaliaguet2010notes}. The agents
are rational and seek to optimize a value function by selecting appropriate
controls as follows:
\begin{equation}
\inf_{\alpha\in\mathbb{A}}J(t,x;\alpha)\hbox{ with }J(t,x;\alpha)=\E\left[\int_t^TL(X^{t,x}_s,\alpha_s,\cL(X^{t,x}_s))ds+g(X^{t,x}_T,\cL(X^{t,x}_T))\right],\label{eq:MFG1}
\end{equation}
where $X^{t,x}_s$ is the solution of
\begin{equation}
\begin{cases}
dX_s&=\si(X_s)\alpha_sds+\si(X_s)\circ dW_s,\ s\in[t,T],\\
X_t&=x\in\R^d,
\end{cases}\label{eq:MFG2a}
\end{equation}
and $\circ$ represents Stratonovich integration.
The dynamics of the private state $X_s$ of a representative agent
is given by \eqref{eq:MFG2a}. We consider a complete probability space $\left(\Omega,\mathcal{F},P\right)$
endowed with an $m-$dimensional Brownian motion $W=\left(W_{t}\right)_{t\in[0,T]}$
and its filtration $\mathbb{F}=\left(\mathcal{F}_{t}\right)_{t\in\left[0,T\right]}$
augmented by all $P-$null sets and a ``rich enough'' sub $\sigma-$algebra
$\mathcal{F}_{0}$ independent of $W$ (see Section \ref{sec:diff-meas}). Furthermore, $\si(x)=\left[\si_1(x),\ldots,\si_m(x)\right]\in\R^{d\times m}$,
and $\si_1,\ldots,\si_m$
are $C^{2,\alpha}_b$ vector fields, $\alpha> 0$, i.e., their second
derivatives have bounded $\alpha$-H\"older norm;
this will allow us to ensure later that the Stratonovich to It\^o
integral correction belongs to $C^{1,\alpha}_b$. 
We will denote $X^{0,x}_s$ by $X^x_s$.

In \eqref{eq:MFG1}--\eqref{eq:MFG2a}, the agent manages its state
by choosing a control $\alpha\in\mathbb{A}$; $\alpha_s$ is a progressively
measurable $\R^m-$valued stochastic process satisfying
the admissibility condition 
\begin{equation}
\E\left[\int_t^T\left|\alpha_r\right|^2dr\right]<\infty\label{eq:sq_int}
\end{equation}
The agent chooses a control driven by the desire of minimizing an
expected cost, $J(t,x;\alpha)$, over a period $[t,T]$.
This expected cost is a combination of a running cost, $L:\R^d\times\R^m\times\cP_2(\R^d)\to\R$,
and a terminal cost, $g:\R^d\times\cP_2(\R^d)\to\R$, where
$\cP_2(\R^d)$ is the space of Borel probability
measures on $\R^d$ with finite second moments and $\cL(X_s)$
stands for the law of $X_s$. Both $L$ and $g$ include the interactions
between the agent and the mean field represented by $\cL(X_s)$.

The fact that the statistical distribution of the agents has to be
given by $\cL(X_s)$ indicates that we are searching
for an equilibrium in the sense of Nash (see, e.g., \cite{carmona2018probabilistic}).

In what follows, we consider the SDE \eqref{eq:MFG2a} in terms of It\^{o} integration, that is, 
\begin{equation}
\begin{cases}
dX_s&=\left[b(X_s)+\si(X_s)\alpha_s\right]dt+\si(X_s)dW_s,~s\in\left[t,T\right],\\
X_t&=x\in\R^d.
\end{cases}.\label{eq:MFG2}
\end{equation}
where
\[
b^i(x)=\frac12\sum_{l=1}^m\sum_{j=1}^d\si_l^j(x)\partial_j\si_l^i(x),~i=1,\ldots,d,
\]
correspond to the Stratonovich to It\^o integral correction.

Since solutions of SDEs like \eqref{eq:MFG2} are expected to have
finite moments, we shall work in the space $\cP_2(\R^d)$
which consists of Borel probability measures with finite second moments,
i.e.,
\[
\cP_2(\R^d)=\left\{ \mu\in\cP_2(\R^d):\int_{\R^d}\left|x\right|^2d\mu(x)<\infty\right\} .
\]
Furthermore, we endow $\cP_2(\R^d)$
with the 2--Wasserstein distance $W_2$: if $\mu,\nu\in\cP_2(\R^d)$,
the 2--Wasserstein distance $W_2(\mu,\nu)$ is given
by:
\begin{equation}
W_2(\mu,\nu)=\inf_{\pi\in\Pi_2(\mu,\nu)}\left[\int_{\R^d\times\R^d}\left|x-y\right|^2\pi(dx,dy)\right]^{1/2},\label{eq:MK}
\end{equation}
where $\Pi_2(\mu,\nu)$ is the set of probability measures
in $\cP_2(\R^d\times\R^d)$
with marginals $\mu$ and $\nu$. Moreover, if $X,X'$ are square
integrable $\R^d-$valued random variables, we have
\[
W_2(\cL(X),\cL(X'))^2\leq \E\left[\left|X-X'\right|^2\right].
\]

For our functionals defined on the 2-Wasserstein space, we will consider differentiation with respect to the probability measure $\mu$ as introduced by P. L. Lions in \cite{cardaliaguet2010notes}: for  (a differentiable) $f:\cP_2(\R^d)\to\R$, the derivative of $f$ with respect to $\mu$ is a function $\partial_\mu{f}:\cP_2(\R^d)\times\R^d\to\R^d$. In particular, we will focus on the space $C^{1,1}_b(\cP_2(\R^d))$ of continuously differentiable functionals over $\cP_2(\R^d)$  with Lipschitz-continuos bounded derivatives; see Section \ref{sec:Preliminaries} for details.

One can study \eqref{eq:MFG1}-\eqref{eq:MFG2} through a probabilistic representation of the value function
of the optimization problem as the solution of a Backward Stochastic
Differential Equation (BSDE). For this, we look for a minimizing control
\[
\hat\af(x,z,\mu)=\underset{a\in\R^m}{\arg\min}\ L(x,a,\mu)-a\cdot z,
~\hbox{ for all }(x,z,\mu)\in\R^d\times\R^m\times\cP_2(\R^d).
\]
We assume that $L(\cdot,\cdot,\mu)$ belongs to
$C^3(\R^d\times\R^m)$ for all $\mu\in\cP_2(\R^d)$ and that $L$ is strongly convex and
has quadratic growth in $a$, which imply the existence of a
unique minimizer, $\hat\af$; by the implicit function theorem, $\hat\af(\cdot,\cdot,\mu)$ belongs to
$C^2(\R^d\times\R^m)$  for all $\mu\in\cP_2(\R^d)$.

Therefore, we can
represent \eqref{eq:MFG1}-\eqref{eq:MFG2} as the following McKean-Vlasov FBSDE
(see \cite[Section 4.4]{carmona2018probabilistic}):
\begin{equation}
\begin{cases}
dX_s=\left[b(X_s)+\si(X_s)\hat{\alpha}(X_s,Z_s,\cL(X_s))\right]ds+\si(X_s)dW_s\\
dY_s=-L(X_s,\hat{\alpha}(X_s,Z_s,\cL(X_s)),\cL(X_s))ds+Z_s\cdot dW_s
\end{cases}\label{eq:Str_MFG}
\end{equation}
for $s\in\left[t,T\right]$, $X_t=x$, and $Y_T=g(X_T,\cL(X_T))$.

In \eqref{eq:Str_MFG}, the process $Z_s$ is referred to as
the {\em control process}. The function $L$ is called the {\em driver}
and the random variable $g(X_T,\cL(X_T))$ is the {\em terminal condition}. 

We present our main object of study in the following section.

\subsection{Formulation of the problem and connection with stochastic differential MFG}
To motivate our problem, we will focus on the following {\em version} of
\eqref{eq:MFG1}-\eqref{eq:MFG2} (see Section 3.3.1 \cite{carmona2018probabilistic}).
Let $(\Omega,\cF,\cF_s^t,P,W)$ be 
our generalized reference probability space
\cite{fabbri2017stochastic}(Definition 1.100) and assume now that $\alpha_r$, 
$r\in\left[t,T\right]$ is an $\cF_s^t-$progressively 
measurable process with values in $\R^m$ such that 

\begin{equation}
  \E\left[\exp(\frac12\int_t^T\left|\alpha_r\right|^2dr)\right]<\infty
  \hbox{ (Novikov's condition).}\label{eq:Nov}
\end{equation}
We will denote
\[H\ast W=\int_t^. H_rdW_r\]
for a process $(H_s)_{s\in[t,T]}$, 
and $\cE(Z)$ the stochastic exponential of a process $Z$.

Let $M=\cE(\af\ast W)$,
since the controls $\alpha$ satisfy Novikov's condition \eqref{eq:Nov}, 
Girsanov's theorem ensures that $M_s$ is a $P-$martingale and 
we can define a probability $\tilde{P}$ by setting $\tilde{P}(A)=\E\left[\mathbf1_{A}M_T\right],$
$A\in\cF$. Moreover, the process 
\[
\tilde{W}_s=W_s-W_t+\int_t^s\alpha_rdr 
\]
is an $m-$dimensional Brownian motion with respect to $\cF_s^t$
and $\tilde{P}$, and 
\begin{align*}
\int_t^s\si(X_r)d\tilde{W}_r= & \int_t^s\si(X_r)dW_r-\left[\int_t^s\si(X_r)dW_r, (H\ast W)_s\right]\\
= & \int_t^s\si(X_r)dW_r+\int_t^s\si(X_r)\alpha_rdr, 
\end{align*}
where $\left[\cdot,\cdot\right]$ is the $\R^d-$ valued 
covariation process. Considering the last expression, \eqref{eq:MFG2}
becomes 
\begin{equation}
dX_s=b(X_s)ds+\si(X_s)d\tilde{W}_s.\label{eq:MFG2b}
\end{equation}
We define 
\begin{equation}
\inf_{\af\in\A}J^{weak}(\alpha)\hbox{ with }J^{weak}(\alpha)=\E^{\tilde{P}}\left[g(X_T,\cL(X_T))+\int_t^TL(X_s,\alpha_t,\cL(X_s))ds\right].\label{eq:weak_cost}
\end{equation}

We associate to \eqref{eq:MFG2b}-\eqref{eq:weak_cost} 
the following FBSDE: 
\begin{equation}
\begin{cases}
dX_s=b(X_s)ds+\si(X_s)d\tilde{W}_s,\\
dY_s=-L(X_s,\hat{\alpha}(X_s,Z_s, \cL(X_s),\cL(X_s))ds+Z_s\cdot d\tilde{W}_s, 
\end{cases}\label{eq:weak_mfg2}
\end{equation}
$s\in\left[t,T\right]$, $X_t=x$, and $Y_T=g(X_T,\cL(X_T))$,  where 
\[
\hat{\alpha}(x,z,\mu)=\underset{a\in\mathbb A}{\arg\min}\left\{ L(x,a,\mu)-a\cdot z\right\}. 
\]

System (\ref{eq:weak_mfg2}) is the main object of study in this paper. As mentioned in Section 9 \cite{carmona2015probabilistic}, the study of regularity of (\ref{eq:MFG2b})-(\ref{eq:weak_cost}) may be based on (\ref{eq:weak_mfg2}). We emphasize that one of the main difficulties in our
analysis is related to the quadratic growth of $L$ in its second
variable. In what follows, we also assume that the functions $L$ and $g$ satisfy the Lasry-Lions monotonicity condition:
\begin{defn}\label{Lions}
A real valued function $B$ on $\R^d\times\cP_2(\R^d)$
is monotone (in the sense of Lasry and Lions) if, for all $\mu\in\cP_2(\R^d)$,
the mapping $\R^d\ni x\mapsto B(x,\mu)$ is at
most of quadratic growth, and, for all $\mu,\mu'\in\cP_2(\R^d)$
we have:
\[
\int_{\R^d}\left[B(x,\mu)-B(x,\mu')\right]d(\mu-\mu')(x)\geq0.
\]
\end{defn}

The rest of the paper is organized as follows.
In Section \ref{sec:result}, we present our results.
In Section \ref{sec:Preliminaries},
we review the necessary theory for our analysis, in particular Malliavin
calculus and differentiability of functionals defined on $\cP_2(\R^d)$.
In Section \ref{sec:Preliminaries2}, we recall the theory of FBSDE
with quadratic growth. Finally, we study
the McKean-Vlasov FBSDE \eqref{eq:weak_mfg2} and provide proofs of our results in Section \ref{sec:proof-results}.

\section{Statement of results}
\label{sec:result}

\begin{assumption}\label{assu:A0}
  We assume that
  \begin{enumerate}[(i)]
   \item The vector fields $\si_1,\ldots,\si_m: \R^d\to\R^d$ are
     $C^{2,\alpha}_b$, $\alpha>0$.
  \end{enumerate}

\begin{enumerate}[(i')]
  \item  The vector fields $\si_1,\ldots,\si_m$ are smooth and  have
    bounded derivatives of all orders.
\end{enumerate}

 \begin{enumerate}[(i)]
 \setcounter{enumi}{1}

\item  $g\in C^{1,1}_b(\R^d\times\cP_2(\R^d))$, see subsection \ref{sec:diff-meas}.

\item The functions $g$ and $L(\cdot,a,\cdot)$ for all $a\in\R^m$,
  satisfy the Lasry-Lions monotonicity condition. See Definition \ref{Lions}.
\end{enumerate}
\end{assumption}
\begin{assumption}
  \label{assu:A1}
  $L:\R^d\times\R^m\times\cP_2(\R^d)\to\R$ is such that
  $L(\cdot,\cdot,\mu)$ belongs to $C^3(\R^d\times\R^m)$ for all $\mu\in\cP_2(\R^d)$
and $L(x,a,\cdot), \nabla L(x,a,\cdot), D_{aa}^2L(x,a,\cdot)$ are
differentiable for all $x\in\R^d$, $a\in\R^m$, see subsection \ref{sec:diff-meas}.
Moreover,  there exist constants $C,\ga>0$ such that for
all $x\in\R^d$, $a,\xi\in\R^m$, $\mu\in\cP_2(\R^d)$ we have
\end{assumption}
\begin{enumerate} 
\item $|L(x,0,\mu)|\le C$, $|\nabla_xL(x,0,\mu)|\le C$, $|\nabla_aL(x,0,\mu)|\le C$.
\item $\xi^TD_{aa}^2L(x,a,\mu)\xi\geq\gamma\left|\xi\right|^2$,
\item $\left|D_{aa}^2L(x,a,\mu)\right|\leq C$,
\item $\left|D_{ax}^2L(x,a,\mu)\right|\leq C(1+\left|a\right|)$,
\item $\left|D_{xx}^2L(x,a,\mu)\right|\leq C(1+\left|a\right|^2)$.
  \item $\left|\partial_\mu (\nabla_aL)(x,a,\mu)\right|\leq C(1+\left|a\right|)$,
\item $\left|\partial_\mu  (\nabla_xL)(x,a,\mu)\right|\leq C(1+\left|a\right|^2)$
\end{enumerate}

Assumption \ref {assu:A1}  implies that there is a constant $C^*$
such that for all $x\in\R^d$, $a\in\R^m$, $\mu\in\cP_2(\R^d)$ we have 
\begin{enumerate}[(a)]
\item   $\dfrac\ga2\left|a\right|^2-C^*\leq L(x,a,\mu)\leq C^*(1+\left|a\right|^2)$,
\item  $\left|\nabla_xL(x,a,\mu)\right|\leq C^*(1+\left|a\right|^2),$
\item $\ga\left|a\right|^2\leq a^T\nabla_{a}L(x,a,\mu)+C^*\left|a\right|$,
\item $\ga\left|a\right|\leq\left|\nabla_{a}L(x,a,\mu)\right|+C^*$,
 $\left|\nabla_{a}L(x,a,\mu)\right| \leq C^*(1+\left|a\right|)$.
\end{enumerate}

\begin{assumption}
  \label{assu:A2}There exist constants $C,\eta>0$ such that for all $x,x'\in\R^d$
$a,a',\xi\in\R^m$, $\mu,\mu'\in\cP_2(\R^d)$   the function
$L:\R^d\times\R^m\times \cP_2(\R^d)\to\R$ satisfies
\begin{align}
\tag{1}
  -C\left|\xi\right|^2\leq D_{aaa}^{3}L((D_{aa}^2L)^{-1}\xi,(D_{aa}^2L)^{-1}\xi,&(D_{aa}^2L)^{-1}D_{a}L)
                                            \le\xi^T(D_{aa}^2L)^{-1}\xi-\eta\left|\xi\right|^2,\\
\tag2 |D_{xaa}^{3}L(x,a,\mu) |&\leq C, \\
  \tag{3}|\partial_\mu(D_{aa}^2L)(x,a,\mu,v)|&\le C,\\
  \tag{4} |\partial_\mu L(x',a',\mu',v') -\partial_\mu L(x,a,\mu,v)|&\\\nonumber
  \le C (1+|a|+|a'|)\{(1+|a|+|a'|)&(|x'-x|+W_2(\mu',\mu) +|v'-v|))+|a'-a|\},\\
  \tag{5} |D_{xa}^2L(x',a',\mu') -D_{xa}^2L(x,a,\mu)|
  \le C&\{(1+|a|+|a'|)(|x'-x|+W_2(\mu',\mu)+|a'-a|\},\\ 
  \tag{6} |\partial_\mu(\nabla_a L)(x',a',\mu',v') -\partial_\mu(\nabla_a L)(x,a,\mu,v)|&\\
  \le C\{(1+|a|+|a'|)&(|x'-x|+W_2(\mu',\mu)+|v'-v|)+|a'-a|\}.\nonumber
\end{align}

\end{assumption}

\begin{rem}
  Assumption \ref{assu:A2} specifies the quadratic behaviour of $L$.
  One can verify that for $f\in C_b^{1,1}(\R^d\times\cP_2(\R^d))$
  with $\inf f>0$, 
 the function $L(x,a,\mu)=f(x,\mu)|a|^2$ satisfies Assumptions \ref{assu:A1} and \ref{assu:A2}.
\end{rem}

To characterize the regularity of the involved stochastic
processes, we use the spaces $\D^{\infty}, \mL^{1,2}$, which
will be defined in subsection \ref{sec:Malliavin}.

\begin{thm}\label{thm:weak_MFG}
  Under Assumptions \ref{assu:A0}, \ref{assu:A1} and
  \ref{assu:A2}, there exists  
a unique solution \\ $(X^{t,x}_s,Y^{t,x}_s,Z^{t,x}_s)$ of \eqref{eq:weak_mfg2}
such that $\cL(X^{t,x}_s)$
is the unique equilibrium of the MFG associated with the stochastic
optimal control problem
\eqref{eq:MFG2b}-\eqref{eq:weak_cost}. Moreover

\begin{enumerate}
  \item 
    There exists a function
$\Omega\times\left[0,T\right]\times\R^d\to\R^d\times\R\times\R^m$,
$(\omega,t,x)\mapsto(X_t^x,Y_t^x,Z_t^x)(\omega)$
, such that for almost every $\omega$, the mappings $(t,x)\mapsto X_t^x(\om)$
and $(t,x)\mapsto Y_t^x(\om)$ are continuous in $t$ and
continuously differentiable in $x$ in the classical sense.

\item Under Assumption \ref{assu:A0}(i') we have that  each component
  of $X^{t,x}_s$ belongs to $\D^{\infty}$.

\item For any $t\in\left[0,T\right]$ and $x\in\R^d$, $(Y_t^x,Z_t^x)\in\mL^{1,2}\times(\mL^{1,2})^m$.
Furthermore, $\left\{ D_tY_t^x;t\in\left[0,T\right]\right\} $
is a version of $\left\{ Z_t^x;t\in\left[0,T\right]\right\} $.

\end{enumerate}
\end{thm}

\section{Review}\label{sec:Preliminaries} 

This section presents the basic ideas and notation used throughout this paper. 

\subsection{Malliavin calculus}
\label{sec:Malliavin}
Let $H$ be a real separable Hilbert space with scalar product and
norm denoted by $\langle \cdot,\cdot\rangle _H$ and
$\|\cdot\| _H$, respectively. Let $  W=\{ W(h),h\in H\} $
be an isonormal Gaussian process defined on a complete probability
space $(\Omega,\cF,P)$. Later on, we will focus on
the case where $(\Omega,\cF,P)$ is the canonical
probability space associated with an $m-$dimensional Brownian motion,
\[\{ W^i(t),t\in[0,T]\}\ i=1,\ldots,m,\
  H=L^2([0,T];\R^m),\hbox{ and }
  W(h)=\sum_{i=1}^m\int_0^Th_t^idW_t^i \]
  (Wiener integral), but we start with the abstract setting to avoid
cluttered notation.

We want to differentiate a square integrable random variable, $F:\Omega\to\R$,
with respect to the chance parameter $\omega\in\Omega$. For this,
let $\cS$ be the class of smooth random variables of the
form
\begin{equation}
F=f(W(h_1),\ldots,W(h_m)),\label{eq:SRV}
\end{equation}
where $f\in C_p^{\infty}(\R^m)$ (the set of
infinitely continuously differentiable functions $f:\R^m\to\R$
such that $f$ and all of its partial derivatives have polynomial
growth), $h_1,\ldots,h_m\in H$, and $m\geq1$. The (Malliavin)
derivative of a smooth random variable, $F$, of the form \eqref{eq:SRV}
is the $H-$valued random variable 
\[
DF=\sum_{i=1}^m\partial_if(W(h_1),\ldots,W(h_m))h_i.
\]

The operator $D$ is closable from $L^p(\Omega)$ to
$L^p(\Om;H)$ for any $p\geq1$ (see \cite[Proposition 1.2.1]{nualart2006malliavin}).
We denote the domain of $D$ in $L^p(\Omega)$ by $\D^{1,p}$,
i.e., $\D^{1,p}$ is the closure of the class $\cS$
with respect to the norm
\[
\left\Vert F\right\Vert _{1,p}=\left[\E(\left|F\right|^p)+\E(\left\Vert DF\right\Vert _H^p)\right]^{\frac1p}.
\]
Moreover, the iterated derivative $D^kF$ is a random variable with
values in $H^{\otimes k}$. For $1\leq k\in\mathbb{N}$, $p\geq1$,
and $F\in\cS$, let
\[
\left\Vert F\right\Vert _{k,p}=\left[\E(\left|F\right|^p)+\sum_{j=1}^k\E(\left\Vert D^jF\right\Vert _{H^{\otimes j}}^p)\right]^{\frac1p}.
\]
$\D^{k,p}$ denotes the completion of $\cS$ with
respect to the norm $\left\Vert \cdot\right\Vert _{k,p}$. In addition,
let 
\[
\D^{k,\infty}=\bigcap_{p>1}\D^{k,p}\hbox{ and }\D^{\infty}=\bigcap_{k>1}\bigcap_{p>1}\D^{k,p}.
\]
The previous definitions can be extended to Hilbert-valued random
variables. 

Note that if $H=L^2(\left[0,T\right];\R^m)$,
then $L^2(\Omega;H)\simeq L^2(\left[0,T\right]\times\Omega;\R^m)$
and hence the derivative of $F\in\D^{1,2}$ is a square integrable
$\R^m-$valued process. 

The adjoint of the operator $D$ is the divergence operator; we denote
it by $\delta$. In particular, $\delta$ is an unbounded operator
on $L^2(\Omega;H)$ with values in $L^2(\Omega)$
such that for all $u\in\hbox{Dom }\delta\subseteq L^2(\Omega;H)$
we have $\left|\E(\left\langle DF,u\right\rangle _H)\right|\leq C\left\Vert F\right\Vert _2$,
for every $F\in\D^{1,2}$, where $C$ depends on $u$. Furthermore,
$\E(F\delta(u))=\E(\left\langle DF,u\right\rangle _H)$.
The space $\D^{1,2}(H)$ is included in the domain
of $\delta$ (see \cite[Proposition 1.3.1]{nualart2006malliavin}.
If $H=L^2(\left[0,T\right])$, then $\hbox{Dom }\delta\subset L^2(\left[0,T\right]\times\Omega)$;
in this case, $\delta(u)$ is called the Skorohod stochastic
integral of the process $u$.

We denote by $\mL^{1,2}$ the space $\D^{1,2}(L^2(\left[0,T\right]))$.
$\mL^{1,2}$ coincides with the class of processes $u\in L^2(\left[0,T\right]\times\Omega)$
such that $u(t)\in\D^{1,2}$ for almost all $t$,
and there exists a measurable version of the two parameter process
$D_su_t$ such that $\E\int_0^T\int_0^t(D_su_t)^2dsdt<\infty$.
By \cite[Proposition 1.3.1]{nualart2006malliavin}, $\mL^{1,2}\subset\hbox{Dom }\delta$.
On the other hand, $\mL^{1,2}$ is a Hilbert space with the
norm
\[
\left\Vert u\right\Vert _{\mL^{1,2}}^2=\left\Vert u\right\Vert _{L^2(\left[0,T\right]\times\Omega)}^2+\left\Vert Du\right\Vert _{L^2(\left[0,T\right]^2\times\Omega)}^2.
\]
If, in addition, we consider an $m-$dimensional Brownian motion,
then the class of square integrable adapted processes (with respect
to the filtration generated by the Brownian motion) belongs to the
domain of $\delta$. Furthermore, $\delta$ restricted to such class
coincides with the It\^{o} integral (see \cite[Proposition 1.3.11]{nualart2006malliavin}).

Later, we will use a generalization of the space $\mL^{1,2}$.
To define it, consider $H=L^2(\left[0,T\right];\R^m)$
along with an $m-$dimensional Brownian motion. Let $\cH^p(\R^d)$
be the space of progressively measurable processes $(X_t)_{t\in\left[0,T\right]}$
with values in $\R^d$ normed by 
\[
\left\Vert X\right\Vert _{\cH^p}=\E\left[\left(\int_0^T\left|X_s\right|^2ds\right)^{p/2}\right]^{\frac1p}.
\]
For $1\leq k\in\mathbb{N}$, $p\geq1$, let $\mL^{k,p}(\R^d)$
be the class of $\R^d-$valued progressively measurable
processes $u=(u^1,\ldots,u^d)$ on $\left[0,T\right]\times\Omega$
such that
\begin{enumerate}
\item $u(t,\cdot)\in(\D^{k,p})^d$ for almost
all $t\in\left[0,T\right]$;
\item $t,\omega\to D_s^ku(t,\omega)\in(L^2(\left[0,T\right]^{k+1}))^{m\times d}$,
$s=(s_1,\ldots,s_k)\in\left[0,T\right]^k$, admits
a progressively measurable version;
\item $\left\Vert u\right\Vert _{\mL^{k,p}}=\left[\left\Vert \left|u\right|\right\Vert _{\cH^p(\R^d)}^p+\sum_{i=1}^k\left\Vert \left|D^iu\right|\right\Vert _{(\cH^p(\R^d))^{i+1}}^p\right]^{1/p}<\infty.$
\end{enumerate}
Finally, let $\cS^p(\R^d)$ be the
space of all measurable processes $(X_t)_{t\in\left[0,T\right]}$
with values in $\R^d$ normed by
\[
\left\Vert X\right\Vert _{\cS^p}=\E\left[\left(\sup_{t\in[0,T]}\left|X_t\right|\right)^p\right]^{^{\frac1p}}.
\]

\subsection{Differentiability of Functions of Probability Measures}
\label{sec:diff-meas}
We consider a ``rich enough'' probability space $\left(\Omega,\mathcal{F},P\right)$,
which means that for every $\mu\in\mathcal{P}_{2}\left(\mathbb{R}^{d}\right)$,
there is a random variable, $\upsilon\in L^{2}\left(\Omega,\mathcal{F},P;\mathbb{R}^{d}\right)$,
such that $\mathcal{L}\left(\upsilon\right)=\mu$; see details in
\cite{buckdahn2017mean} Section 2. 

A function $f:\cP_2(\R^d)\to\R$
is differentiable at $\mu\in\cP_2(\R^d)$
(in the sense of Lions) if, for $\tilde{f}(\xi):=f(\cL(\xi))$,
$\xi\in L^2(\cF;\R^d)$, there is some
$\xi_0\in L^2(\cF;\R^d)$ with $\cL(\xi_0)=\mu$,
such that the function $\tilde{f}:L^2(\cF,\R^d)\to\R$
is (Frechet) differentiable at $\xi_0$, that is, there exists a
linear continuous mapping, $D\tilde{f}(\xi_0):L^2(\cF;\R^d)\to\R$
such that
\[
\tilde{f}(\xi_0+\eta)-\tilde{f}(\xi_0)=D\tilde{f}(\xi_0)(\eta)+o(\| \eta\| _{L^2}),
\]
with $\| \eta\| _{L^2}\to0$ for $\eta\in L^2(\cF;\R^d)$.
By the Riesz representation theorem, there is a ($P-$a.s.) unique
random variable, $\Theta_0\in L^2(\cF;\R^d)$,
such that
\[
D\tilde{f}(\xi_0)(\eta)=\langle \Theta_0,\eta\rangle _{L^2}=E[\Theta_0\cdot\eta],
\]
for all $\eta\in L^2(\cF;\R^d)$. It was proved in
\cite{cardaliaguet2010notes} Theorem 6.5 that there is  $l_0\in L^2_\mu(\R^d,\R^d)$
such that $\Theta_0=l_0(\xi_0)$, $P-$a.s. Therefore,
we can write
\[
f(\cL(\xi))-f(\cL(\xi_0))=E[l_0(\xi_0)\cdot(\xi-\xi_0)]+o(\| \xi-\xi_0\| _{L^2}),
\]
$\xi\in L^2(\cF;\R^d)$.

We define 
\[
\partial_\mu f(\cL(\xi_0),y):=l_0(y),~y\in\R^d,
\]
the derivative of $f:\cP_2(\R^d)\to\R$
at $\cL(\xi_0)$. $\partial_\mu f(\cL(\xi_0),y)$
is $\cL(\xi_0)(dy)-$ a.e. uniquely
determined. 

Since we have to consider functions $f:\cP_2(\R^d)\to\R$
which are differentiable over the whole space $\cP_2(\R^d)$,
we will assume that $\tilde{f}:L^2(\cF;\R^d)\to\R$
is Frechet differentiable in all of $L^2(\cF;\R^d)$.
In this case, $\partial_\mu f(\cL(\xi),y)$
is defined $\cL(\xi)(dy)-$a.e. for
all $\xi\in L^2(\cF;\R^d)$. Furthermore,
Lemma 3.3\cite{carmona2015forward} shows that if the Frechet derivative $D\tilde{f}$
is Lipschitz continuous (Lipschitz constant $K$), then there exists
for every $\xi\in L^2(\cF;\R^d)$ an
$\cL(\xi)-$version of $\partial_\mu f(\cL(\xi),\cdot):\R^d\to\R^d$
with
\[
|\partial_\mu f(\cL(\xi),y)-\partial_\mu f(\cL(\xi),y')|\leq K|y-y'|\hbox{ for all }y,y'\in\R^d.
\]

In \cite{buckdahn2017mean}, these results motivate the following
\begin{defn}\label{c11mu}\quad
  \begin{enumerate}[(a)]
  \item  $f\in C_b^{1,1}(\cP_2(\R^d))$ 
if  for all $\xi\in L^2(\cF;\R^d)$ there exists
an $\cL(\xi)-$modification of $\partial_\mu f(\cL(\xi),\cdot)$,
again denoted by $\partial_\mu f(\cL(\xi),\cdot)$,
such that $\partial_\mu f:\cP_2(\R^d)\times\R^d\to\R^d$
is bounded and Lipschitz continuous, that is, there is a constant $C$ such that, for $v,v'\in\R^d$,
and   $\mu,\mu'\in\cP_2(\R^d)$ we have
\begin{align*}
|\partial_\mu f(\mu,v)|&\leq C\\
|\partial_\mu f(\mu,v)-\partial_\mu f(\mu',v')|&\leq C(W_2(\mu,\mu')+|v-v'|)
\end{align*}

We consider this function $\partial_\mu f$
as the derivative of $f$. 
\item Let $g:\R^d\times\cP_2(\R^d)\to\R$ be such that for any
  $\mu\in\cP_2(\R^d)$, $g(\cdot,\mu)$ is differentiable and for any
  $x\in\R^d$,  $g(x,\cdot)$ is differentiable.
We say that $g\in C_b^1(\R^d\times\cP_2(\R^d))$ if $\nabla_xg$, $\partial_\mu g$
are continuous and bounded.
We say that $g\in C_b^{1,1}(\R^2\times\cP_2(\R^d))$,
if there $C$ is a constant, such that for $x,x',v,v'\in\R^d$,
  $\mu,\mu'\in\cP_2(\R^d)$ we have
  \begin{align*}
    |\nabla_xg(x,\mu)-\nabla_xg(x',\mu')|&\le C(|x-x'|+W_2(\mu,\mu'))\\
 |\partial_\mu g(x,\mu,v)-\partial_\mu g(x',\mu',v')|&\le C(|x-x'|+W_2(\mu,\mu') +|v-v'|)
  \end{align*}
 \end{enumerate}
\end{defn}

\section{Preliminary results}\label{sec:Preliminaries2} 
In Section \ref{sec:proof-results}, we will study the McKean-Vlasov system 
\begin{align}
X_t^x= & x+\int_0^tb(X_s^x)dr+\int_0^t\si(X_s^x)\ dW_s,\label{eq:SDE}\\
Y_t^x= & g(X^x_T,\cL(X^x_T))-\int_t^TZ_s^x\cdot
         dW_s+\int_t^TF(s,X_s^x,Z_s^x,\cL(X_s^x)) ds,\label{eq:BSDE}
\end{align}
where $b$, $\si$, $g$, and $W$ are as in section \ref{sec:result}
and $F:[0,T]\times\R^d\times\R^m\times\cP_2(\R^d)\to\R$
satisfies Assumption \ref{assu:A3} below. We follow closely the theory in \cite{dos2010some} for the the related (non McKean-Vlasov) FBSDE, which we recall in this section. 

\subsection{Basic  results on SDE}

We have the following results for the SDE \eqref{eq:SDE};
see Theorem 3.1 in \cite{dudley1984stochastic}, Section 2.2 in \cite{nualart2006malliavin}, and Section 7.5 in \cite{nualart2018introduction} for details.
\begin{thm}[Existence and moment estimates]
  Let Assumption \ref{assu:A0}(i) hold, then \eqref{eq:SDE} has a unique solution 
and for any $p\geq2$, $x\in\R^d$, $s,t\in[0,T]$, we have
\begin{equation}
\E[\sup_{t\in[0,T]}|X^x_t|^p]\leq C (1+|x|^p),\label{eq:D_1.2.3}
\end{equation}
\begin{equation}
\E[\sup_{u\in[s,t]}|X^x_u-X^x_s|^p]\leq C (1+|x|^p)|t-s|^{p/2}.\label{eq:D_1.2.4}
\end{equation}
Moreover, given $x,x'\in\R^d$ we have
\begin{equation}
\E[\sup_{t\in[0,T]}|X_t^x-X_t^{x'}|^p]\leq C|x-x'|^p.\label{eq:D_1.2.5}
\end{equation}
\end{thm}

\begin{thm}
  [Classical differentiability]\label{thm:c_diff}
 Let Assumption \ref{assu:A0}(i) hold, then the solution process
 $X_t^x$  of \eqref{eq:SDE}  is continuously differentiable
as a function of the initial condition $x$.
Let $\nabla_xX_t^x$ be the Jacobian matrix, 
then
\[
\nabla_xX_t^x=I_d+\int_0^t\nabla_x\si(X_s^x)\cdot\nabla_xX_s^x\
dW_s+\int_0^t\nabla_xb(X_s^x)\nabla_xX_s^x\ ds.
\]
Furthermore, for any $p\ge 2$
\begin{align}\label{eq:D1.2.6}
\sup_{x\in\R^d}\|\nabla_xX^x\| _{\cS^p} \leq C_p,\\
\E[\sup_{s\le u\le t}\|\nabla_xX_u^x-\nabla_xX_s^x|^{p}] & \leq C_p|t-s|^{p/2},\label{eq:D1.2.7}
\end{align}
and given $x,x'\in\R^d$ we have
\begin{equation}
  \E[\sup_{0\leq t\leq T}|\nabla_xX_t^x-\nabla_xX_t^{x'}|^{p}]\leq C_{p}|x-x'|^{p/2}.\label{eq:D_1.2.7}
\end{equation}
Moreover, $\nabla_xX_t^x$ as an $d\times d$
matrix is invertible for any $t\in\left[0,T\right]$ and its inverse
$(\nabla_xX_t^x)^{-1}$ satisfies an SDE.
\end{thm}

\begin{thm}
  [Malliavin differentiability]\label{thm:Malliavin}
  Let Assumption \ref{assu:A0}(i) hold and let $X^x_t$ be the solution of the stochastic
    differential equation \eqref{eq:SDE}. Then each component $X^i(t)$
  of $X^{t,x}_s$ belongs to $\D^{\infty}$.
  Moreover, 
\begin{equation}
\underset{0\leq r\leq t}{\sup}\mathbb{E}\left(\underset{r\leq s\leq T}{\sup}\left|D_{r}^{j}X^{i}\left(s\right)\right|^{p}\right)<\infty.\label{eq:M_der_bound}
\end{equation}
The Malliavin derivative admits a version $(u,t) \mapsto D_uX_t$ which satisfies an SDE. By Theorem \ref{thm:c_diff} we have the representation
\begin{equation}
D_{u}X^x_{t}=\nabla_{x}X^x_{t}(\nabla_{x}X^x_{u})^{-1}\sigma(X^x_{u})1_{\left[0,u\right]}(t)\hbox{ for all }u,t\in\left[0,T\right];\label{eq:rep1.2.10}
\end{equation}
see eq. (2.59) in \cite{nualart2006malliavin}.
\end{thm}

\subsection{Results on Lipschitz BSDE}

In this and the following subsection, we recall relevant results for
our computations concerning moment 
estimates for BSDE with drivers that satisfy Lipschitz conditions
with random Lipschitz constant. For more details, particularly the
BMO (Bounded Mean Oscillation) property, see Section 1.2.5 in \cite{dos2010some}.

Let $\psi:\Om\times[0,T]\times\R\times\R^d\to\R$ be a measurable function and let $\zeta$ be a random
variable. Consider the BSDE
\begin{equation}
U_t=\zeta-\int_t^TV_sdW_s+\int_t^T\psi\left(\cdot,s,U_s,V_s\right)ds,\quad
t\in\left[0,T\right].\label{eq:BSDE_lip}
\end{equation}
For $p\geq1$, assume
\begin{description}
\item [(A1)] $\zeta$ is an $\cF_T$-adapted random variable
and $\zeta\in L^{2p}(\R)$.
\item [(A2)] $\psi:\Om\times[0,T]\times\R\times\R^d\to\R$
is measurable and there exists a positive constant, $M$,
and a positive predictable process, $(H_t)_{t\in[0,T]}$,
such that for all $t\in[0,T],$ $u,u'\in\R$, and
$v,v'\in\mathbb{R}^m$, we have
\begin{equation}
|\psi(\cdot,t,u,v)-\psi(\cdot,t,u',v')|\leq M|u-u'|+H_t|v-v'|.\label{eq:lm_bsde_a2}
\end{equation}
Furthermore, $H\ast W$ is a BMO martingale.
\item [(A3)] $\left(\psi\left(\cdot,t,0,0\right)\right)_{t\in\left[0,T\right]}$
is a measurable $\cF_t$-adapted process such that for every
$p\geq1$ we have $\displaystyle\E\left[\Bigl(\int_0^t|\psi(\cdot,s,0,0)|ds\Bigr)^{p}\right]<\infty$.
\end{description}

We have the following result (Lemma 2.1.1 in \cite{dos2010some}):
\begin{lem}
  \label{lem:Mom_bsde_lip}
  Assume (A1)-(A3). Let $p\geq1$ and $\bar{r}>1$
be such that $\cE(H\ast W)\in L^{\bar{r}}\left(\mathbb{P}\right)$.
Assume that the pair $(U,V)$ is a square integrable
solution of (\ref{eq:BSDE_lip}).
Then, there exists a positive constant, C, depending only on $p,$
$T,$ $M$, and the BMO norm of $H\ast W$, such that, with the conjugate
exponent $\bar{q}$ of $\bar{r}$, we have
\[
\|U\|_{\cS^{2p}}^{2p}+\|V\|_{\cH^{2p}}^{2p}\leq
C\E\left[|\zeta|^{2p\bar{q}^2}+\Bigl(\int_0^t|\psi(\cdot,s,0,0)|\
    ds\Bigr)^{2p\bar{q}^2}\right]^{1/q^2}.
\]
\end{lem}

\subsection{Results on FBSDE\label{subsec:FBSDE}}

Basic results can be found in
\cite{ankirchner2007classical,ankirchner2010pricing,dos2010some,dos2011some,imkeller2010path}.
As in \cite{dos2010some}, we consider the parameterized BSDE
\begin{equation}
  \label{BSDE}
Y_t^x= \xi(x)-\int_t^TZ_s^x\cdot dW_s+\int_t^Tf(s,x,Z_s^x)\ ds
\end{equation}
under the condition
\begin{description}
\item[(C1)]
$f : \Om \times [0, T] \times \R^d \times \R^m \to \R$ is an
adapted measurable function, differentiable in the spatial variables with
continuous partial derivatives. There is a positive constant $M$ and a
positive process $(K_t(x))_{t\in[0,T ]}$ depending on $x\in \R^d$ such that for all
$(t, x, z)  \in[0, T ] \times\R^d\times\R^m$
\begin{align*}
  |f (t, x, z)| &\le M (1 + |z|^2 )  \quad\hbox{a.s.},\\
|\nabla_x f (t, x, z)|& \le K_t(x)(1 + |x| + |z|^2 ) \quad\hbox{a.s.},\\
|\nabla_z f (t, x, z)|& \le M (1 + |z|) \quad\hbox{a.s.}\\
\sup_{x\in\R^d}\|K(x)\|_{\cS^{2p}} &<\infty  \quad\hbox{for any }
  p\ge 1.
\end{align*}
For any $x \in\R^d$, $\xi(x)$ is a random variable 
$\cF_T$-adapted such that $\sup\limits_{x\in\R^d}\|\xi(x)\|_{L^\infty} < \infty$;
for all $p\ge 1$ the map $\R^d\mapsto L^{2p}$ ,
$x\mapsto \xi(x)$ is differentiable, its derivative $\nabla_x\xi$
belongs to $L^{2p}(\R^d)$, is continuous and
$\sup\limits_{x\in\R^d}\|\nabla_x\xi(x)\|_{L^{2p}}<\infty$.
\end{description}
\begin{thm}[Differentiability]\label{difDosReis}\cite{dos2010some} Theorem 3.1.3.
  Assume (C1) holds. Then for all $x\in\R^d$ and $p\ge 1$
  (\ref{BSDE}) has a unique solution $(Y^x_.,Z^x_.)\in\cS^p(\R)\times\cH^{2p}(\R^m)$.
Moreover the map $\R^d\to\cS^{2p}(\R)\times\cH^{2p}(\R^m)$,
$x\mapsto(Y^x,Z^x)$ is differentiable and the derivative is a solution
of
\begin{equation}
  \label{derYparam}
  \nabla Y_t^x= \nabla\xi(x)-\int_t^T\nabla Z_s^x\cdot dW_s+
  \int_t^T(\nabla_xf(s,x,Z_s^x)+\nabla_zf(s,x,Z_s^x)\nabla  Z_s^x)\ ds.
  \end{equation}
\end{thm}
Consider the BSDE
\begin{equation}
  \label{dos3.2.1}
  Y_t = \xi - \int_t^TZ_sdW_s + \int_t^T f(s,Z_s)ds,\quad t \in  [0,T]
\end{equation}
under the following assumptions
\begin{description}
\item[(E1)] $f : \Om\times[0,T]\times\R^m\to\R$ is an adapted
measurable function continuously differentiable in the spatial
variable. There exists a positive constant $M$ such that for all
$(t,z) \in  [0, T ] \times \R^m$
\begin{align*}
  |f(t, z)| \le M(1 + |z|^2) \hbox { a.s.}\\
|\nabla_z f (t, z)| \le M (1 + |z|) \hbox { a.s.}
\end{align*}
\item[(E2)]  For each $z \in  \R^m$, $(f(t,z))_{t\in [0,T]} \in\mL
  ^{1,2p}(\R)$ for all $p\ge 1$ and its Malliavin derivative is given by
  $(D_uf(t,z))_{u,t\in [0,T]}$. For each $u,t \in[0,T]$, $z\mapsto
  D_uf(t,z)$ is continuous. There exist two positive adapted processes
  $(K_u(t))_{u,t\in [0,T ]}$ and $(\tilde K_u(t))_{u,t\in[0,T]}$
  satisfying for all $p\ge  1$
\[\int_0^T\E[\|K_u(t)\|_{\cH^{2p}}^{2p} +\|\tilde K_u(t)\|_{\cS^{2p}}^{2p}] du < \infty\]
such that for any $(u,t,z)\in [0,T]\times[0,T]\times\R^m$ we have
\[|D_uf(t,z))|\le K_u(t)(1+|z|)+\tilde K_u(t)|z|^2, \hbox { a.s}\]
\item [(E3)] $\xi$ is a $\cF_T$-adapted measurable random variable
  absolutely bounded and belongs to $\D^{1,\infty}$.
\end{description}
\begin{thm}\label{paramMalliavin}\cite{dos2010some} Theorem 3.2.3.
  Assume that $f$ and $\xi$ satisfy (E1), (E2) and (E3). Then, the
  solution process $(Y,Z)$ of \eqref{dos3.2.1} belongs to
  $\mL^{1,2}\times(\mL^{1,2})^m$, and a version of
  $(D_uY_t,D_uZ_t)_{u, t\in[0,T]}$ is the unique solution of
  \[
D_uY_t=0,\ D_uZ_t=0,\ t\in[0,u),
\]
\[D_uY_t= D_u\xi-\int_t^TD_uZ_s\ dW_s+\int_t^T[(D_uf)(s,Z_s)+\nabla_z
  f(s,Z_s)D_uZ_s]ds,\; t\in[u,T].\]
Moreover, $\{ D_tY_t;t\in[0,T]\} $, is a version of $\{ Z_t;t\in[0,T]\} $.
\end{thm}
\section{Proof of results}
\label{sec:proof-results}

Consider the McKean-Vlasov system \eqref{eq:SDE}-\eqref{eq:BSDE} under the following assumption:

\begin{assumption}
  \label{assu:A3}\quad
\begin{description}
\item [(HY0)] 
  $g:\R^d \times\cP_2(\R^d)\to \R$ is measurable
continuous and uniformly bounded by a positive constant
$K$; $F:[0,T]\times\R^d\times\R^m\times\cP_2(\R^d)\to \R$
is measurable and continuous in the spatial variables.
There exists a positive constant, $M$, such that for
all $t\in[0,T]$, $x,x'\in\R^d$, and $z,z\in\R^m$, $\mu, \mu'\in\cP_2(\R^d)$
we have
\begin{align*}
|F(t,x,z,\mu)|\leq & M(1+|z|^2),\\
|F(t,x,z,\mu)-F(t,x',z,\mu')|\leq & M(1+|z|^2)(|x-x'|+W_2(\mu,\mu')),\\
|F(t,x,z,\mu)-F(t,x,z',\mu)|\leq & M (1+|z|+|z'|)|z-z'| .
\end{align*}
\item [(HY1)] (HY0) holds, $g\in C_b^1(\R^d\times\cP_2(\R^d))$,
for each $(t,\mu)\in [0,T]\times\cP_2(\R^d)$, $F(t, \cdot, \cdot,\mu)\in C^1(\R^d\times\R^m)$ and for each $(t, x, z)\in [0,T]\times\R^d\times\R^m$,
$F(t,x, z, \cdot)$ is differentiable.
Furthermore, there exists a positive constant, $M$, such that for all
$(t,x,z,\mu,v)\in[0,T]\times\mathbb{\R}^d\times\R^m\times\R^d$ 
we have
\begin{align*}
|\nabla_xF(t,x,z,\mu)|\leq & M(1+|z|^2),\\
|\nabla_z F(t,x,z,\mu)|\leq & M(1+|z|)\\
|\partial_\mu F(t,x,z,\mu,v)|\leq & M(1+|z|^2)
\end{align*}
\item [(HY1+)] (HY1) holds, $g\in C_b^{1,1}(\R^d\times\cP_2(\R^d))$. 
For $x,x',v,v'\in\R^d$,
  $z,z'\in\R^m$  and   $\mu,\mu'\in\cP_2(\R^d)$ we have
  \begin{equation}
    \label{hy22}
    |\nabla_{z}F(t,x,z,\mu)-\nabla_{z}F(t,x',z',\mu')|\leq M
  \{(1+|z|+|z'|)(|x-x'|+W_2(\mu,\mu'))+|z-z'|\} 
  \end{equation}
\begin{align*}
  |\nabla_xF(t,x,z,\mu)-\nabla_xF(t,x',z',\mu')|&\\
  \leq M(1+|z|+|z'|)\{ (1+|z|+|z'|)&(|x-x'|+W_2(\mu,\mu'))+|z-z'|\} 
  \\
  |\partial_\mu F(t,x,z,\mu,v)-\partial_\mu F(t,x',z',\mu',v')|&\\
  \leq M(1+|z|+|z'|)\{ (1+|z|+|z'|)&(|x-x'|+W_2(\mu,\mu')+|v-v'|)+|z-z'|\}.
  \end{align*}
\end{description}
\end{assumption}
We will use the notation $\mathbb{E}^{\omega}\left[\cdot\right]$
to emphasize that the expectation acts only on specific elements.
We denote by $\Theta^x_t=(X^x_t,Z^x_t)$ the solution of \eqref{eq:SDE}- \eqref{eq:BSDE},  and consider the following linear FBSDE system: 
\begin{align}\label{derX}
  \nabla X^x_t&=I_d+\int_0^t(\nabla_x\si(X_s)\cdot\nabla_xX_s)\cdot dW_s+\int_0^t\nabla_xb(X_s)\nabla_xX_sds,\\
  \label{derY}
\nabla Y^x_t&=\nabla_xg(X^x_T,\cL(X^x_T)) \nabla X^x_T+\E^\om[\partial_\mu g(X^x_T,\cL(X^x_T), X^x_T(\om))\nabla X^x_T(\om)]\\\nonumber
       &-\int_t^T\nabla Z^x_s dW_s+\int_t^T(\nabla_xF,\nabla_zF)(s,\Te^x_s,\cL(X^x_s)) \nabla\Te^x_sds\\
   \nonumber           \ &+\int_t^T\E^\om[\partial_\mu F(s,\Te^x_s,\cL(X^x_s), X^x_s(\om))\nabla X^x_s(\om)]ds.
\end{align}

\begin{thm}\label{TFBSDE1}
  Let Assumptions \ref{assu:A0}(i), \ref{assu:A3} hold.
  Then, for $x\in\R^d$ and $p\geq1$, 
  \eqref{eq:SDE}-\eqref{eq:BSDE} has a unique solution
$(X_.^x,Y_.^x,Z_.^x)\in\cS^{2p}(\R^d)\times\cS^{2p}(\R)\times\cH^{2p}(\R^m)$,
and the map $\R^d\to\cS^{2p}(\R^d)\times\cS^{2p}(\R)\times\cH^{2p}(\R^m)$,
$x\mapsto(X_t^x,Y_t^x,Z_t^x)$ is differentiable and the derivative is a solution of 
\eqref{derX}-\eqref{derY}.
\end{thm}
\begin{proof}
  
We will check that Assumption \ref{assu:A3} implies condition
(C1) for the terminal condition $\xi(x)=g(X^x_T, \cL(X^x_T))$ and
the driver $f:\Om\times[0,T]\times\R^d\times\R^m\to\R$ given
by $f(\om,t,x,z)=F(t,X_t^x(\om),z,\cL(X^x_t))$. Thus, we can appeal to
Theorem \ref{difDosReis} to conclude differentiability.

We start with the statement about the terminal condition.
We have that $\xi(x)$ is $\cF_T$-adapted and 
\[\nabla\xi(x)=\nabla_xg(X^x_T,\cL(X^x_T)) \nabla X^x_T+\E^\om[\partial_\mu g(X^x_T,\cL(X^x_T), X^x_T(\om))\nabla X^x_T(\om)].\]
Thus $x\mapsto\nabla\xi(x)$ is continuous and
\[|\nabla\xi(x)|\le K (|\nabla X^x_T|+\E[|\nabla X^x_T|])\le K (|\nabla X^x_T|+\|\nabla X^x\|_{\cS^p})\]
Using inequality \eqref{eq:D1.2.6} we obtain
$\sup\limits_{x\in\R^d}\|\nabla\xi(x)\|_{L^p}<\infty$ for any $p\ge 2$.

The continuity of
$x\mapsto X_t^x$ combined with (HY1) yields that $f$
and $\nabla_zf$ are continuous in $x$ and from (HY1) there is $M>0$ such that
for all $(t,x,z)\in[0,T]\times\R^d\times\R^m$ we have that
\begin{align*}
  |f(t,x,z)|\leq & M(1+|z|^2)\\
|\nabla_zf(t,x,z)|\leq & M(1+|z|).
\end{align*}
 Since
\[\nabla_xf(t,x,z)=\nabla_xF(t,X^x_t,z,\cL(X^x_t))\nabla X^x_t+
  \E^\om[\partial_\mu F(t,X^x_t,z,\cL(X^x_t), X^x_t(\om))\nabla X^x_t(\om)]\]
$(t,x,z))\mapsto\nabla_xf(t,x,z)$ is continuous and from (HY1) there is $M>0$ such that  
for all $(t,x,z)\in[0,T]\times\R^d\times\R^m$ 
\begin{align*}
  |\nabla_xf(t,x,z)|&\le M(1+|z|^2) (|\nabla X^x_t+\E[|\nabla X^x_t|])
                   \le M(1+|z|^2) (|\nabla X^x_t|+\|\nabla X^x\|_{\cS^p})\\
                 :& =K_t(x)(1+|z|^2) 
\end{align*}
  From \eqref{eq:D1.2.6}  we have that $\sup\limits_{x\in\R^d}\|K(x)\|_{\cS^p}<\infty$
for any $p\ge 2$.
\end{proof}

\begin{thm}
  [Classical differentiability]\label{TFBSDE2}
  Let Assumptions \ref{assu:A0}(i),  \ref{assu:A3} hold.
  For $x\in\R^d$
  and $p\geq1$, let $(X^x_.,Y^x_.,Z^x_.)\in\cS^{2p}(\R^d)\times\cS^{2p}(\R)\times\cH^{2p}(\R^m)$
be the solution of \eqref{eq:SDE}-\eqref{eq:BSDE}. Then, there exists a function
$\Omega\times\left[0,T\right]\times\R^d\to\R^d\times\R\times\R^m$,
$(\omega,t,x)\mapsto(X^x_t,Y^x_t,Z^x_t)(\omega)$,
 such that for almost every $\omega$, the mappings $(t,x)\mapsto X_t^x$
and $(t,x)\mapsto Y_t^x$ are continuous in $t$ and
continuously differentiable in $x$.
\end{thm}

In the proof we will use the following 2 Lemmas.  
\begin{lem}\label{terminal}
  Assume (HY1+) holds. For $x\in\R^d$
let $(X^x_.,Y^x_.,Z^x_.)\in\cS^{2p}(\R^d)\times\cS^{2p}(\R)\times\cH^{2p}(\R^m)$
be the solution of \eqref{eq:SDE}-\eqref{eq:BSDE}. Then, for every
$p\ge 1$ there
exists a constant $C>0$, such that for all $x,x'\in\R^d$ we have
\begin{multline*}
  \E\Big[|\nabla_xg(X^x_T,\cL(X^x_T)) \nabla X^x_T
    -\nabla_xg(X^{x'},\cL(X^{x'})) \nabla X^{x'}_T|^{2p} \Big]+\\
\E\Big[\big|\E^\om[\partial_\mu g(X^x_T,\cL(X^x_T), X^x_T(\om))\nabla X^x_T(\om)]
  -\E^\om[\partial_\mu g(X^{x'},\cL(X^{x'}_T), X^{x'}_T(\om)) \nabla X^{x'}_T(\om)]\big|^{2p}\Big]\\\le
C|x-x'|^p.
\end{multline*}
\end{lem}
\begin{proof}
  We have
  \begin{multline*}
    |\nabla_xg(X^x_T,\cL(X^x_T))-\nabla_xg(X^{x'},\cL(X^{x'}))|\\
    \le |\nabla X^{x'}_T||\nabla_xg(X^x_T,\cL(X^x_T))-\nabla_xg(X^{x'},\cL(X^{x'}_T))|+
    |\nabla_xg(X^x_T,\cL(X^x_T))||\nabla X^{x'}_T-\nabla X^x_T|,
  \end{multline*}
  as well as
  \begin{multline*}
   \E\Big[\big( |\nabla
     X^{x'}_T||\nabla_xg(X^x_T,\cL(X^x_T))-\nabla_xg(X^{x'},\cL(X^{x'}))|\big)^{2p}\Big]\\\le
     M^{2p}\E\Big[ |\nabla
       X^{x'}_T|^{2p}\big(|X^x_T-X^{x'}_T|+W_2(\cL(X^x),\cL(X^{x'}))|\big)^{2p}\Big]\le C|x-x'|^{2p}.
     \end{multline*}
    and
   \[\E\Big[ |\nabla_xg(X^x_T,\cL(X^x_T))|^{2p}|\nabla X^{x'}_T-\nabla X^x_T|^{2p}\Big]
     \le K^{2p}\E\big[|\nabla X^{x'}_T-\nabla X^x_T|^{2p}\big]\le C|x-x'|^p.\]
Also
   \begin{multline*}
     \big|\E^\om[\partial_\mu g(X^x_T,\cL(X^x_T), X^x_T(\om))\nabla X^x_T(\om)]
  -\E^\om[\partial_\mu g(X^{x'},\cL(X^{x'}), X^{x'}_T(\om)) \nabla X^{x'}_T(\om)]\big|\\\le
 \big|\E^\om[(\partial_\mu g(X^{x'}_T,\cL(X^{x'}_T), X^{x'}_T(\om))
  -\partial_\mu g(X^x_T,\cL(X^x_T), X^x_T(\om)))\nabla X^{x'}_T(\om)]\big|\\+
  \big|\E^\om[\partial_\mu g(X^x_T,\cL(X^x_T), X^x_T(\om))(\nabla X^{x'}_T(\om)
  -\nabla X^x_T(\om))]\big|\\
\le M\E^\om\big[(|X^{x'}_T-X^x_T|+W_2(\cL(X^x_T),\cL(X^{x'}_T))+|X^{x'}_T(\om)-X^x_T(\om)|) |\nabla X^{x'}_T(\om)|\big]\\
+K \E^\om\big[|\nabla X^{x'}_T(\om)-\nabla X^x_T(\om)|\big]\le C
(|X^{x'}_T-X^x_T|+|x-x'|+|x-x'|^\frac12),
\end{multline*}
and thus
\begin{multline*}
  \E\Big[\big|\E^\om[\partial_\mu g(X^x_T,\cL(X^x_T), X^x_T(\om))\nabla X^x_T(\om)]
  -\E^\om[\partial_\mu g(X^{x'},\cL(X^{x'}_T), X^{x'}_T(\om)) \nabla X^{x'}_T(\om)]\big|^{2p}\Big]\\\le
C\big[\E[|X^{x'}_T-X^x_T|^{2p}]+|x-x'|^{2p}+|x-x'|^p\big]\le C[|x-x'|^{2p}+|x-x'|^p].
\end{multline*}
\end{proof}
\begin{lem}\label{driver}
  Assume (HY1+) holds. For $x\in\R^d$
let $(X^x,Y^x,Z^x)\in\cS^{2p}(\R^d)\times\cS^{2p}(\R)\times\cH^{2p}(\R^m)$
be the solution of \eqref{eq:SDE}-\eqref{eq:BSDE} and $(\nabla X^x,\nabla
Y^x,\nabla Z^x)$ be the solution of \eqref{derX}-\eqref{derY}.
Then, for every
$p\ge 1$ there exists a constant $C>0$, such that for all $x,x'\in\R^d$ we have
\begin{align*}
  \E\Big[\Big(\int_0^T&|\nabla_xF(s,\Te^x_s,\cL(X^x_s)) \nabla X^x_s
    -\nabla_xF(s,\Te_s ^{x'},\cL(X^{x'}_s)) \nabla X^{x'}_s|ds\Big)^{2p} \Big]+\\
\E\Big[\Big(\int_0^T&\big|\E^\om[\partial_\mu F(s,\Te^x_s,\cL(X^x_s), X^x_s(\om))\nabla X^x_s(\om)]
  \\&-\E^\om[\partial_\mu F(s,\Te^{x'},\cL(X^{x'}_s), X^{x'}_s(\om)) \nabla X^{x'}_s(\om)]\big|)ds\Big)^{2p}\Big]\le
C|x-x'|^p
\end{align*}
\end{lem}
\begin{proof}
  We have
  \begin{multline*}
 |\nabla_xF(s,\Te^x_s,\cL(X^x_s))\nabla X^x_s-\nabla_xF(s,\Te_s ^{x'},\cL(X^{x'}_s))\nabla X^{x'}_s|\\\le
|\nabla_xF(s,\Te^x_s,\cL(X^x_s))-\nabla_xF(s,\Te_s ^{x'},\cL(X^{x'}_s))||\nabla X^{x'}_s|+\\
|\nabla_xF(s,\Te^x_s,\cL(X^x_s))||\nabla X^x_s-\nabla X^{x'}_s|,
\end{multline*}
as well as
\begin{multline*}
 \E\Big[\Big(\int_0^T|\nabla_xF(s,\Te^x_s,\cL(X^x_s))
    -\nabla_xF(s,\Te_s ^{x'},\cL(X^{x'}_s))||\nabla
    X^{x'}_s|\Big)^{2p}\Big]\\\le C
\|\nabla
X^x\|_{\cS^{6p}}^{2p}\{\|1+|Z^x|+|Z^{x'}|\|_{\cH^{12p}}^{4p}\|X^x-X^{x'}\|_{\cS^{6p}}^{2p}\\+\|1+|Z^x|+|Z^{x'}|\|_{\cH^{4p}}^{4p}\sup_{s\in[0,T]}W_2(\cL(X^x_s), \cL(X^{x'}_s))
+\|1+|Z^x|+|Z^{x'}|\|_{\cH^{6p}}^{2p}\|Z^x-Z^{x'}\|_{\cH^{6p}}^{2p}\}\\\le C|x-x'|^{2p},
  \end{multline*}
and
  \begin{multline*}
    \E\Big[\Big(\int_0^T |\nabla_xF(s,\Te^x_s,\cL(X^x_s))||\nabla
    X^x_s-\nabla X^{x'}_s|ds\Big)^{2p}\Big]\\\le C
\E\Big[\sup_{s\in[0,T]}|\nabla X^x_s-\nabla X^{x'}_s|^{2p}
\Big(\int_0^T(1+|Z^x|^2)ds\Big)^p\Big]\\\le C
\|\nabla X^x-\nabla X^{x'}\|_{\cS^{6p}}^{2p}(1+\|Z^{x'}\|_{\cH^{8p}}^{4p})
\le C|x-x'|^p.
\end{multline*}
Also
\begin{multline*}
\big|\E^\om[\partial_\mu F(s,\Te^x_s,\cL(X^x_s), X^x_s(\om))\nabla X^x_s(\om)]
 -\E^\om[\partial_\mu F(s,\Te^{x'},\cL(X^{x'}_s), X^{x'}_s(\om))\nabla X^{x'}_s(\om)]\big|\\\le
  \E^\om[|\partial_\mu F(s,\Te^x_s,\cL(X^x_s), X^x_s(\om))
  -\partial_\mu F(s,\Te^{x'},\cL(X^{x'}_s), X^{x'}_s(\om))||\nabla X^{x'}_s(\om)|]
+\\ \E^\om[|\partial_\mu F(s,\Te^x_s,\cL(X^x_s), X^x_s(\om))||\nabla X^x_s(\om)-\nabla X^{x'}_s(\om)|]\\\le
M (1+|Z^x_s|+|Z^{x'}_s|)\E^\om\big[|\nabla X^{x'}_s(\om)|\\
  \{(1+|Z^x_s|+|Z^{x'}_s|)(|X^x_s-X^{x'}_s|+W_2(\cL(X^x_s),\cL(X^{x'}_s))
  +|X^x_s(\om)-X^{x'}_s(\om)|)+|Z^x_s-Z^{x'}_s|\}\big]\\
  +K (1+|Z^x_s|^2)\E^\om\big[|\nabla X^{x'}_s(\om)-\nabla X^x_s(\om)|\big]\\
 \le C (1+|Z^x_s|+|Z^{x'}_s|)\sup_x\|\nabla X^x\|_{\cS^2}\\
\{(1+|Z^x_s|+|Z^{x'}_s|)(|X^x_s-X^{x'}_s|+\E[\sup_{s\in[0,T]}|X^x_s-X^{x'}_s|])+|Z^x_s-Z^{x'}_s|\}\\+
C (1+|Z^x_s|^2)\E\big[\sup_{s\in[0,T]}|\nabla X^x_s-\nabla X^{x'}_s|^2\big]^\frac12
 \\
\le C (1+|Z^x_s|+|Z^{x'}_s|) \{(1+|Z^x_s|+|Z^{x'}_s|)(|X^x_s-X^{x'}_s|+|x-x'|)+|Z^x_s-Z^{x'}_s|\}\\
+C(1+|Z^x_s|^2)|x-x'|^\frac12,
\end{multline*}
and therefore
\begin{multline*}
  \E\Big[\Big(\int_0^T\big|\E^\om[\partial_\mu F(s,\Te^x_s,\cL(X^x_s), X^x_s(\om))\nabla X^x_s(\om)]
  \\-\E^\om[\partial_\mu F(s,\Te^{x'},\cL(X^{x'}_s), X^{x'}_s(\om)) \nabla X^{x'}_s(\om)]\big|)ds\Big)^{2p}\Big]\\\le
C  \E\Big[\Big(\int_0^T (1+|Z^x_s|+|Z^{x'}_s|)^2ds\Big)^{2p}\sup_{s\in[0,T]}|X^x_s-X^{x'}_s|^{2p}\Big]\\+C|x-x'|^{2p}\E\Big[\Big(\int_0^T (1+|Z^x_s|+|Z^{x'}_s|)^2ds\Big)^{2p}\Big]\\+C\E\Big[\Big(\int_0^T (1+|Z^x_s|+|Z^{x'}_s|)^2ds\Big)^p\Big(\int_0^T (|Z^x_s-Z^{x'}_s|)^2ds\Big)^p\Big]\\
+C|x'-x|^p\E\Big[\Big(1+\int_0^T Z^x_s|^2ds)\Big)^{2p}\\
\le C \big(
\|1+|Z^x|+|Z^{x'}|\|_{\cH^{8p}}^{4p}\|X^x-X^{x'}\|_{\cS^{4p}}^{2p}+\|1+|Z^x|+|Z^{x'}|\|_{\cH^{4p}}^{4p}|x'-x|^{2p}\big)\\
+ C \big(\|1+|Z^x|+|Z^{x'}|\|_{\cH^{4p}}^{2p}\|Z^x-Z^{x'}\|_{\cH^{4p}}^{2p}+(1+\|Z^{x'}\|_{\cH^{8p}}^{4p})|x'-x|^p\big)\\\le
C(|x-x'|^{2p}+|x-x'|^p)\le C|x-x'|^p
  \end{multline*}
\end{proof}
\begin{proof}[Proof of Theorem \ref{TFBSDE2}].
We now use Lemas \ref{terminal}, \ref{driver} and the Kolmogorov's continuity criterion
to prove the continuity of $x\mapsto\nabla_x Y^x_t$.
Let $p\ge 1$, $t \in[0,T]$. Take $x,x' \in\R^d$ and let  $(X^x,
Y^x,Z^x)$, $(X^{x'}, Y^{x'},Z^{x'})$ be
the solutions of FBSDE \eqref{eq:SDE}-\eqref{eq:BSDE} with parameters $x$ and $x'$
respectively. Now define $\de\nabla X = \nabla X^x-\nabla X^{x'}$,
$\de\nabla Y = \nabla Y^x-\nabla Y^{x'}$ and $\de\nabla Z = \nabla Z^x-\nabla Z^{x'}$.
We apply Kolmogorov's continuity criterion to
the difference $|\nabla Y^x - \nabla Y^{x'}|$.
We can write a BSDE for the differences $\de\nabla Y_t$, $\de\nabla Z_t$,
$t\in[0,T]$,
\begin{multline*}
    \de\nabla Y_t=\nabla_xg(X^x_T,\cL(X^x_T))-\nabla_xg(X^{x'},\cL(X^{x'}))
  +\E^\om[\partial_\mu g(X^x_T,\cL(X^x_T), X^x_T(\om))\nabla X^x_T(\om)]\\
  -\E^\om[\partial_\mu g(X^{x'}_T,\cL(X^{x'}_T), X^{x'}_T(\om))\nabla X^{x'}_T(\om)]-\int_t^T\de\nabla Z_s dW_s\\
  +\int_t^T[\nabla_xF(s,\Te^x_s,\cL(X^x_s))\nabla X^x_s- \nabla_xF(s,\Te^{x'}_s,\cL(X^{x'}_s))\nabla X^{x'}_s]ds\\
  +\int_t^T[(\nabla_zF(s,\Te^x_s,\cL(X^x_s))-\nabla_zF(s,\Te^{x'}_s,\cL(X^{x'}_s)))\nabla Z^x_s
  +\nabla_zF(s,\Te^{x'}_s,\cL(X^{x'}_s))\de\nabla Z^x_s]ds\\
  +  \int_t^T\E^\om[\partial_\mu F(s,\Te^x_s,\cL(X^x_s),X^x_s(\om))\nabla X^x_s(\om)
   -\partial_\mu F(s,\Te^{x'}_s,\cL(X^{x'}_s), X^{x'}_s(\om))\nabla X^{x'}_s(\om)]ds
 \end{multline*}
 We now apply Lemma \ref{lem:Mom_bsde_lip} to this BSDE to obtain
 \begin{multline*}
   \E[\sup_{t\in[0,T]}|\nabla Y_t^x-\nabla Y_t^{x'}|^{2p}]\\
\le C \Big\{\E \Big[|\nabla_xg(X^x_T,\cL(X^x_T)) \nabla X^x_T -\nabla_xg(X^{x'},\cL(X^{x'})) \nabla X^{x'}_T|^{2p\bar q^2} \\
+\big|\E^\om[\partial_\mu g(X^x_T,\cL(X^x_T), X^x_T(\om))\nabla X^x_T(\om)
  -\partial_\mu g(X^{x'},\cL(X^{x'}_T), X^{x'}_T(\om)) \nabla X^{x'}_T(\om)]\big|^{2p\bar q^2}\\
+\Big(\int_0^T(|\nabla_xF(s,\Te^x_s,\cL(X^x_s)) \nabla X^x_s
    -\nabla_xF(s,\Te_s ^{x'},\cL(X^{x'}_s)) \nabla X^{x'}_s|+\\
\big|\E^\om[\partial_\mu F(s,\Te^x_s,\cL(X^x_s), X^x_s(\om))\nabla X^x_s(\om)
  -\partial_\mu F(s,\Te^{x'},\cL(X^{x'}_s), X^{x'}_s(\om)) \nabla
  X^{x'}_s(\om)]\big|\\+|\nabla_zF(s,\Te^x_s,\cL(X^x_s))-\nabla_zF(s,\Te^{x'}_s,\cL(X^{x'}_s))||\nabla Z^x_s|
  )ds\Big)^{2p\bar q^2}\Big]^\frac 1{\bar q^2}\Big\}
\end{multline*}
Applying Lemmas \ref{terminal} and \ref{driver} we have that there
is $C>0$ such that
\[ \E\Big[\sup_{t\in[0,T]}|\nabla Y_t^x-\nabla Y_t^{x'}|^{2p}\Big]\le C|x-x'|^p.\]
Choosing $p$ large enough and applying Kolmogorov's continuity
criterion we obtaning the continuity of $x\mapsto \nabla Y^x$.
\end{proof}

\begin{thm}[Malliavin differentiability]\label{mainMalliavin}
  Let 
  Assumptions \ref{assu:A0}(i), \ref{assu:A3} hold.
  For $x\in\R^d$, let $(X^x,Y^x,Z^x)$ be
 the unique solution of \eqref{eq:SDE})-\eqref{eq:BSDE}. Then, for
 any $t\in\left[0,T\right]$, $x\in\R^d$, $(Y^x,Z^x)\in\mL^{1,2}\times(\mL^{1,2})^m$
and a version of $\{(D_u^iY_t^x,D_u^iZ_t^x) :0\leq u,t\le T\}$, $i=1,\ldots,m$,
is the unique solution of 
\[
D_uY_t^x=0,~D_uZ_t^x=0,~t\in[0,u),
\]
\begin{align}\nonumber
D_uY_t^x&=  \nabla_xg(X^x_T,\cL(X^x_T)) D_uX_T^x-\int_t^TD_uZ_s^xdW_s\\
 & +\int_t^T(\nabla_x F, \nabla_z F)(s, X_s^x,Z_s^x,\cL(X_s^x))(D_uX_s^x,D_uZ_s^x)\
   ds,\quad t\in[u,T]. \label{DYBSDE}
\end{align}
Furthermore, $\{ D_tY_t^x;t\in[0,T]\} $,
defined by the last equation, is a version of $\{ Z_t^x;t\in[0,T]\} $.
\end{thm}

\begin{proof}For fixed $x\in\R^m$ define $f:\Om\times\R^m$ by
  $f(\om,t,z)=F(t,X_t^x(\om),z,\cL(X_t^x))$,
  \[\nabla f(t,z)=\nabla_zF (t,X_t^x,z,\cL(X_t^x)).\]
  From (HY1) we get (E1) and moreover there is $M>0$ such that
  \[|D_uf(t,z)|=|\nabla_xF (t,X_t^x,z,\cL(X_t^x))D_uX_t^x|\le M
    |D_uX_t^x|(1+|z|^2)\ \hbox{ a.s.}\]
By (\ref{eq:M_der_bound}), $\sup\limits_{u\in[0,T]}\|D_uX^x\|_{\cS^{2p}}<\infty$ for any
$p\ge 1$ and $x\in\R^d$ and then we obtain (E2).

Letting $\xi(x)=g(X^x_T, \cL(X^x_T))$ for $x\in\R^m$, we have that $\xi(x)$ is
$\cF_T$-adapted and 
\[|D_u\xi(x)|=|\nabla_xg(X^x_T,\cL(X^x_T))  D_uX^x_T|\le K  |D_uX^x_T|.\]
By Theorem \ref{thm:Malliavin} and inequality \eqref{eq:D_1.2.3} we
have for any $p\ge 2$ 
\[\sup_{u\in [0,T ]} C\E[(1+|X_T^x|)^p |D_uX_T^x|^p] \le
C\E[(1+|X_T^x|)^{2p}]^\frac 12\sup_{u\in [0,T ]}\E[|D_uX_T^x|^{2p}]^\frac 12<\infty\]
and we obtain (E3). Applying Theorem \ref{paramMalliavin} we get the
first part of the Theorem.

For the second part of the Theorem fix $x\in\R^d$ and $0\le u\le t\le T$.
The representation formula for $D_uX^x$ is given by
(\ref{eq:rep1.2.10}). From Theorem \ref{paramMalliavin} we have that
$\{ D_tY_t;t\in[0,T]\} $, is a version of $\{ Z_t;t\in[0,T]\} $.
\end{proof}

\begin{prop}
\label{prop:Exist}Under Assumptions \ref{assu:A0} and \ref{assu:A1},
there exists a unique solution
$(X^{t,x}_s,Y^{t,x}_s,Z^{t,x}_s)\in\cS^2(\R^d)\times\cS^2(\R)\times\cH^2(\R^m)$ 
of \eqref{eq:Str_MFG} such that $(\cL(X^{t,x}_s)$
is the unique equilibrium of the MFG associated with the stochastic
optimal control problem \eqref{eq:MFG1}-\eqref{eq:MFG2}. 
\end{prop}

\begin{proof}
Theorem 4.44 along with remark 4.50 in \cite{carmona2018probabilistic}
ensure that \eqref{eq:Str_MFG} is {\em solvable}. Uniqueness of the
associated MFG problem follows from \cite[Theorem 3.29]{carmona2018probabilistic}.
\end{proof}
\begin{proof}[Proof of Theorem \ref{thm:weak_MFG}]
  Let
\[\hat{\af}(x,z,\mu)=\underset{a\in A}{\arg\min}\ L(x,a,\mu)-a\cdot z,\]
and
\[  H(x,z,\mu)=
  \inf\limits_{a\in A}\ L\left(x,a,\mu\right)-a\cdot z.\]
Under our assumptions, $\hat{\alpha}(x,z)$ is the unique solution of
$\nabla_{a}L(x,\hat{\af},\mu)=z$, i.e., for each $x,\mu$, $\hat{\af}(x,\cdot)$
is the inverse function of $\zeta(x,\cdot,\mu)=\nabla_{a}L(x,\cdot,\mu)$.
Furthermore, by the implicit function theorem, $\hat\af(x,z,\mu)$
is continuously differentiable in all its arguments and property (d) of $L$ implies that
$|\hat\af(x,z,\mu)|\le C(1+|z|)$. Moreover denoting
$G=(D_{aa}^2L)^{-1}$ and writing $\hat\af$ for $\hat{\alpha}(x,z,\mu)$
\begin{align*}
  D_x\hat\af &=-G(x,\hat\af,\mu) D_{xa}^2L(x,\hat\af,\mu)\\
D_z\hat\af=&G(x,\hat\af,\mu)\\
  \partial_\mu\hat{\af}(x,z,\mu,v) &=-G(x,\hat\af,\mu)\partial_\mu(\nabla_aL)(x,\hat\af,\mu,v) 
\end{align*}
Thus
\begin{align*}
  |D_x\hat\af| &\le C(1+|\hat\af|)\le C(1+|z|),\\
|D_z\hat\af|&\le C,\\
| \partial_\mu\hat{\af}(x,z,\mu,v)|&\le C(1+|\hat\af|)\le C(1+|z|)
\end{align*}
which imply
\[|\hat\af(x',z',\mu')-\hat\af(x,z,\mu)|\le C\{(1+|z|+|z'|)(|x'-x|+W_2(\mu',\mu))+|z-z'|\}\]
Let $X_s,Y_s,Z_s$ be the processes given in Proposition \ref{prop:Exist} and
$\hat{\alpha}_s=\hat{\alpha}\left(X_s,Z_s\right)$. Since
the process $Z\ast W$
is a BMO martingale (see \cite[Theorem 4.19]{carmona2018probabilistic}),
so is $\hat\af\ast W$. By \cite[Proposition 4.18]{carmona2018probabilistic},
the stochastic exponential $\cE(\hat\af\ast W)$ is a true martingale. Therefore, we can use Girsanov's theorem to
define the $m-$dimensional Brownian motion
\[
\tilde{W}_s=W_s-W_t+\int_t^s\hat{\alpha}_{r}dr.
\]

One proves as in Proposition \ref{prop:Exist} the existence of a
solution $(X,Y,Z)$  to \eqref{eq:weak_mfg2}. 

Let $F(x,z,\mu):=L(x,\hat{\alpha}(x,z,\mu),\mu)$, 
and consider the system
\begin{equation}\label{eq:weak_mfg2-b}
  \begin{cases}
    dX_s=b(X_s)ds+\si(X_s)dW_s \\
dY_s=-F(X_s,Z_s,\cL(X_s))ds+ Z_s\cdot   dW_s
\end{cases}
\end{equation}
for $s\in\left[t,T\right], X_t=x, Y_T=g(X_T,\cL(X_T)). $
We will show that Assumption \ref{assu:A3} holds, which allows us to apply Theorems \ref{TFBSDE1}, \ref{TFBSDE2}, and \ref{mainMalliavin} and, therefore, obtain the conclusions of Theorem \ref{thm:weak_MFG}. In the following estimates, $C$ denotes a constant that may vary from line to line. Property (a) of $L$ implies
$|F(x,0,\mu)|\le C$.

\begin{equation}
  \nabla_{z}F(x,z,\mu)=\nabla_aL(x,\hat\af,\mu)\nabla_z\hat\af
  =z^TG(x,\hat\af,\mu),\label{eq:D_eta}
\end{equation}
so that
\[\nabla_{z}F(x,z,\mu)D_{aa}^2L(x,\hat\af,\mu)=z\]
\begin{align*}
  \nabla_xF(x,z,\mu)&=\nabla_xL(x,\hat\af,\mu)+\nabla_aL(x,\hat\af,\mu)\nabla_x\hat\af\\
  &=\nabla_xL(x,\hat\af,\mu) -\nabla_zF(x,z,\mu)\nabla_{xa}^2L(x,\hat\af,\mu),\\
    \partial_\mu F(x,z,\mu,v)&=\partial_\mu L(x,\hat\af,\mu)+ \nabla_aL(x,\hat\af,\mu)\partial_\mu\hat\af\\
&  =\partial_\mu L(x,\hat\af,\mu,v) -\nabla_zF(x,z,\mu)\partial_\mu(\nabla_aL)(x,\hat\af,\mu,v).
\end{align*}
Thus
\begin{align*}
  |\nabla_xF(x,z,\mu)|&\le C(1+|\hat\af|^2+|z|(1+|\hat\af|))\le C(1+|z|^2),\\
|\partial_\mu F(x,z,\mu)|&\le C(1+|\hat\af|^2+|z|(1+|\hat\af|))\le C(1+|z|^2).
\end{align*}
Also
\begin{align*}
  D_{zz}^2F(x,z,\mu)&D_{aa}^2L(x,\hat\af,\mu)+
  \nabla_{z}F(x,z,\mu)\cdot D_{aaa}^ 3L(x,\hat\af,\mu) \nabla_z\hat\af=I,\\
D_{zz}^2F(x,z,\mu)&=G(x,\hat\af,\mu)-D_{aaa}^ 3L(x,\hat\af,\mu) [G(x,\hat\af,\mu)z, G(x,\hat\af,\mu),G(x,\hat\af,\mu)].\\
D_{xz}^2F(x,z,\mu)&D_{aa}^2L(x,\hat\af,\mu)+
  \nabla_{z}F(x,z,\mu)(D_{xaa}^ 3L(x,\hat\af,\mu)+ D_{aaa}^ 3L(x,\hat\af,\mu) D_x\hat\af=0,\\
  D_{xz}^2F(x,z,\mu)&=D^2_{xa}L(x,\hat\af,\mu)(G(x,\hat\af,\mu)
 -D_{zz}^2F(x,z,\mu))\\&-\nabla_zF(x,z,\mu)D_{xaa}^3L(x,\hat\af,\mu)G(x,\hat\af,\mu).\\
\partial_\mu\nabla_xF(x,z,\mu,v)&D_{aa}^2L(x,\hat\af,\mu)+
  \nabla_{z}F(x,z,\mu)\partial_\mu(D_{aa}^2L)(x,\hat\af,\mu,v)+ D_{aaa}^ 3L(x,\hat\af,\mu)\partial_\mu\hat\af=0,\\
  \partial_\mu\nabla_xF(x,z,\mu,v)&=\partial_\mu(\nabla_aL)(x,\hat\af,\mu,v)(G(x,\hat\af,\mu)
 -D_{zz}^2F(x,z,\mu))\\&- \nabla_{z}F(x,z,\mu)\partial_\mu(D_{aa}^2L) (x,\hat\af,\mu,v)G(x,\hat\af,\mu).
\end{align*}
Thus
\begin{equation}
\xi^t D_{zz}^2F(x,z,\mu)\xi=\xi^tG(x,\hat\af,\mu)\xi-D_{aaa}^ 3L(x,\hat\af,\mu) [G(x,\hat\af,\mu)z, G(x,\hat\af,\mu)\xi,G(x,\hat\af,\mu)\xi]\label{eq:F1}
\end{equation}
for $\xi\in\R^m$. On the other hand, since $C\left|\xi\right|^2\geq\xi^TD_{aa}^2l\ \xi\geq\gamma\left|\xi\right|^2$,
we have $C\geq\lambda_1,\ldots,\lambda_m\geq\gamma$, where $\lambda_1,\ldots,\lambda_m$
are the eigenvalues of $D_{aa}^2L$. It follows that the eigenvalues
of $G$ are bounded between $1/C$ and
$1/\gamma$. Therefore, under Assumptions \ref{assu:A1}, \ref{assu:A2} we have
\begin{align}\label{}
  |D_{zz}^2F(x,z,\mu)|&\le C.\\
  |D_{xz}^2F(x,z,\mu)|&\le C(1+|z|).\\
|\partial_\mu\nabla_xF(x,z,\mu)|&\le C(1+|z|).
\end{align}
These imply \eqref{hy22} .
By Assumptions \ref{assu:A1}, \ref{assu:A2} \and  \eqref{hy22}
\begin{align*}
  |\nabla_xF(x',z',\mu')-\nabla_xF(x,z,\mu)|\le|\nabla_xL(x',\hat{\alpha}(x',z',\mu'),\mu)- \nabla_xL(x,\hat\af(x,z,\mu),\mu)|\\
  +|\nabla_zF(x',z',\mu')-\nabla_zF(x,z,\mu)||D_{xa}^2L(x',\hat{\alpha}(x',z',\mu'),\mu')|\\
  +|\nabla_zF(x,z,\mu)||D_{xa}^2L(x',\hat{\alpha}(x',z',\mu'),\mu')-D_{xa}^2L(x,\hat\af(x,z,\mu),\mu)|\\
 \leq C(1+|z|+|z'|)\{ (1+|z|+|z'|(|x-x'|+W_2(\mu,\mu'))+|z-z'|\},
\end{align*}
and
\begin{multline*}
  |\partial_\mu F(x',z',\mu',v')- \partial_\mu F(x,z,\mu,v)|\\
  \le|\partial_\mu L(x',\hat{\alpha}(x',z',\mu'),\mu',v')-\partial_\mu L(x,\hat\af(x,z,\mu),\mu,v)|\\ 
  +|\nabla_zF(x',z',\mu')-\nabla_zF(x,z,\mu)||\partial_\mu(\nabla_aL)(x',\hat{\alpha}(x',z',\mu'),\mu',v')|\\
  +|\nabla_zF(x,z,\mu)||\partial_\mu(\nabla_aL)(x',\hat{\alpha}(x',z',\mu'),\mu',v')-\partial_\mu(\nabla_aL)(x,\hat\af(x,z,\mu),\mu,v)|\\
 \leq C(1+|z|+|z'|)\{ (1+|z|+|z'|(|x-x'|+W_2(\mu,\mu'))+|z-z'|\}.
\end{multline*}
\end{proof}

\section{Conclusions}
\label{sec:conclude}

In this paper, our primary contribution is Theorem \ref{thm:weak_MFG}, which proves the classical and Malliavin differentiability of solutions to the stochastic differential MFG system \eqref{eq:weak_mfg2} with quadratic-growth running cost. These regularity results are fundamental to understanding the system's physical meaning and ensuring numerical simulations are accurate. Although Theorem \ref{thm:weak_MFG} addresses the non-fully coupled system \eqref{eq:weak_mfg2}, it lays the groundwork for future analysis of regularity in the fully coupled case.

\subsection*{Author contributions} All authors contributed equally to this manuscript.

\section{Declarations}
\subsection*{Competing interests}
The authors have no competing interests to disclose.
\subsection*{Funding}
The authors acknowledge financial support from CONACYT Mexico.
\subsection*{Availability of data and materials}
Data sharing is not applicable to this article as no new data were created or analyzed in this study.


\end{document}